\documentclass[reqno]{amsart}

\usepackage{amsmath}
\usepackage{amsfonts}
\usepackage{amssymb,enumerate}
\usepackage{amsthm}
\usepackage[all]{xy}
\usepackage{rotating}
\usepackage{hyperref}
\usepackage{color}

\theoremstyle{plain}
\newtheorem{lem}{Lemma}[section]
\newtheorem{cor}[lem]{Corollary}
\newtheorem*{mainthm*}{Main Theorem}
\newtheorem{prop}[lem]{Proposition}
\newtheorem{thm}[lem]{Theorem}

\newtheorem{intthm}{Theorem}

\theoremstyle{definition}

\newtheorem{disc}[lem]{Remark}

\newtheorem{para}[lem]{}

\newtheorem*{convention*}{Convention}

\newcommand{\pd}{\operatorname{pd}}	
\newcommand{\gdim}{\mathrm{G}\text{-}\!\dim}

\newcommand{\id}{\operatorname{id}}

\newcommand{\depth}{\operatorname{depth}}	
\newcommand{\rank}{\operatorname{rank}}

\newcommand{\soc}{\operatorname{Soc}}

\newcommand{\coker}{\operatorname{Coker}}

\newcommand{\im}{\operatorname{Im}}

\newcommand{\End}{\operatorname{End}}

\newcommand{\Ker}{\operatorname{Ker}}

\newcommand{\ideal}[1]{\mathfrak{#1}}
\newcommand{\m}{\ideal{m}}

\newcommand{\fm}{\ideal{m}}

\newcommand{\xra}{\xrightarrow}

\newcommand{\f}{\mathbf{f}}

\renewcommand{\geq}{\geqslant}
\renewcommand{\leq}{\leqslant}
\renewcommand{\ker}{\Ker}

\newcommand{\Ext}[4][R]{\operatorname{Ext}_{#1}^{#2}(#3,#4)}

\newcommand{\Hom}{\operatorname{Hom}}	
\newcommand{\Tor}[4][R]{\operatorname{Tor}^{#1}_{#2}(#3,#4)}

\def\Tor{\operatorname{Tor}}
\def\Ext{\operatorname{Ext}}

\def\soc{\operatorname{Soc}}

\def\diag{\operatorname{diag}}

\numberwithin{equation}{lem}

\begin{document}

\bibliographystyle{amsplain}

\title[Structure of irreducible homomorphisms to/from free modules]{Structure of irreducible homomorphisms \\to/from free modules}

\author{Saeed Nasseh}
\address{Department of Mathematical Sciences\\
Georgia Southern University\\
Statesboro, GA 30460, USA}
\email{snasseh@georgiasouthern.edu}
\urladdr{https://cosm.georgiasouthern.edu/math/saeed.nasseh}

\author{Ryo Takahashi}
\address{Graduate School of Mathematics\\
Nagoya University\\
Furocho, Chikusaku, Nagoya, Aichi 464-8602, Japan}
\email{takahashi@math.nagoya-u.ac.jp}
\urladdr{http://www.math.nagoya-u.ac.jp/~takahashi/}

\thanks{Takahashi was partly supported by JSPS Grants-in-Aid
for Scientific Research 16K05098.}




\begin{abstract}
The primary goal of this paper is to investigate the structure of irreducible monomorphisms to and irreducible epimorphisms from finitely generated free modules over a noetherian local ring.
Then we show that over such a ring, self-vanishing of Ext and Tor for a finitely generated module admitting such an irreducible homomorphism forces the ring to be regular.
\end{abstract}

\maketitle


\section{Introduction}

\begin{convention*}
In this paper, $(R,\fm,k)$ is a commutative noetherian local ring and all modules are finitely generated.
\end{convention*}

A homomorphism $f\colon M\to N$ of $R$-modules is called \emph{irreducible} if
$f$ is neither a split monomorphism nor a split epimorphism, and
for every factorization $M\xra{g}L\xra{h}N$ of $f$ we have
$g$ is a split monomorphism or $h$ is a split epimorphism.
Irreducible homomorphisms are used in the theory of Auslander-Reiten sequences which was established in~\cite{AR1} and play a central role in representation theory of artin algebras. (Excellent references on these topics are~\cite{ARS, roger, yoshino}.)

In this paper we investigate the structure of irreducible monomorphisms to and irreducible epimorphisms from free modules over a commutative noetherian local ring. Section~\ref{sec05062016a} deals with the case where we have an irreducible monomorphism to a free module. Our main result in this section, stated next, is proven in~\ref{para20170617a} and~\ref{para20170617b}.

\begin{intthm}\label{thm16052016a}
The following assertions hold for the local ring $R$.
\begin{enumerate}[\rm(a)]
\item
If $I$ is a non-zero proper ideal of $R$ that is a direct summand of $\m$, then the inclusion map $I\to R$ is an irreducible homomorphism.

\item
Assume that $R$ is Henselian, and let $\phi\colon M\to F$ be an irreducible monomorphism of $R$-modules with $F$ free such that $\im(\phi)\subseteq \m F$. Then the following hold.
\begin{enumerate}[\rm(b1)]
\item
The $R$-module $M$ is isomorphic to a direct summand of $\m$.
\item
If $M$ is indecomposable, then for every surjection $\pi\colon F\to R$ there exists a split monomorphism $\eta\colon M\to \fm$ such that the diagram
$$
\xymatrix{
M\ar[r]^{\phi}\ar[d]_{\eta}&F\ar[d]^{\pi}\\
\fm\ar[r]^{\theta}&R
}
$$
commutes, where $\theta$ stands for the inclusion map.
\end{enumerate}
\end{enumerate}
\end{intthm}

Section~\ref{sec05062016c} is devoted to the case where we have an irreducible epimorphism from a free module. We prove our main result in this section, stated next, in~\ref{para20170617c}.
In this theorem, $\soc R$ denotes the socle of $R$.

\begin{intthm}\label{thm16052016b}
Let $\phi\colon F\to M$ be an irreducible epimorphism of $R$-modules with $F$ free. Then the following assertions hold.
\begin{enumerate}[\rm(a)]
\item
The kernel of $\phi$ is isomorphic to $k$.
\item
Assume that $\End_R(M)$ is a local ring, and let $\iota\colon R\to F$ be a split monomorphism. Then the following hold.
\begin{enumerate}[\rm(b1)]
\item
The composition $\phi\iota\colon R\to M$ is irreducible and hence, it is either surjective or injective.
\item
Suppose that $\phi\iota$ is surjective.
Then $R$ has type one, $F$ has rank one, and there is a commutative diagram
$$
\xymatrix{
F\ar[rr]^{\phi}&&M\\
R\ar[rr]^{\pi}\ar[u]^{\iota}&&R/\soc R\ar[u]_{\rho}
}
$$
such that $\iota$ and $\rho$ are isomorphisms and $\pi$ is the natural surjection.
\end{enumerate}
\end{enumerate}
\end{intthm}

Our motivation for the main result in Section~\ref{sec05062016d} comes from the Auslander-Reiten Conjecture~\cite{AR} that originates in representation theory of artin algebras.
This section deals with this conjecture and also
with a Tor version of it when the module admits irreducible homomorphisms
described in Theorems~\ref{thm16052016a} and~\ref{thm16052016b}; see Theorem~\ref{cor230915a}.

\section{Basic properties}\label{sec05062016b}

This section contains some results that will be used in the subsequent sections. The next result is a part of~\cite[Lemma 5.1]{ARS} in which $R$ is assumed to be an artin algebra. Here we give the proof (with no such assumption on $R$) for the reader's convenience.

\begin{prop}\label{si}
Let $f\colon M\to N$ be an irreducible homomorphism of $R$-modules.
Then $f$ is either surjective or injective.
\end{prop}

\begin{proof}
The map $f$ has a factorization $M\xra{g}\im f\xra{h}N$, where $g$ is the surjection induced by $f$ and $h$ is the inclusion map.
Since $f$ is irreducible, either $g$ is a split monomorphism or $h$ is a split epimorphism.
In the first case, $g$ is an isomorphism, which means that $f$ is injective.
In the second case, $h$ is an isomorphism, which means that $f$ is surjective.
\end{proof}

\begin{lem}\label{p}
Let $f\colon M\to N$ be a homomorphism of $R$-modules.
Then the following are equivalent:
\begin{enumerate}[\rm(i)]
\item
$f$ is irreducible.
\item
$\left(\begin{smallmatrix}
f&0\\
0&1
\end{smallmatrix}\right)\colon M\oplus X\to N\oplus X$ is irreducible for all $R$-modules $X$.
\item
$\left(\begin{smallmatrix}
f&0\\
0&1
\end{smallmatrix}\right)\colon M\oplus X\to N\oplus X$ is irreducible for some $R$-module $X$.
\end{enumerate}
\end{lem}

\begin{proof}
It is straightforward to see that $f$ is neither a split monomorphism nor a split epimorphism if and only if
$\left(\begin{smallmatrix}
f&0\\
0&1
\end{smallmatrix}\right)$ is neither a split monomorphism nor a split epimorphism.

For an arbitrary $R$-module $X$ consider a factorization
$$
M\oplus X\xra{\left(\begin{smallmatrix}\alpha&\beta\end{smallmatrix}\right)}L\xra{\left(\begin{smallmatrix}\gamma\\ \delta\end{smallmatrix}\right)}N\oplus X
$$
of $\left(\begin{smallmatrix}
f&0\\
0&1
\end{smallmatrix}\right)$.
Then $\gamma\alpha=f$, $\gamma\beta=0$, $\delta\alpha=0$, and $\delta\beta=1$.
If $f$ is irreducible, then either $\alpha$ is a split monomorphism or $\gamma$ is a split epimorphism.

If $\alpha$ is a split monomorphism, then there is a homomorphism $\varepsilon\colon L\to M$ such that $\varepsilon\alpha=1$ and we have
$$
\begin{pmatrix}
1&-\varepsilon\beta\\
0&1
\end{pmatrix}\left(\begin{matrix}\varepsilon\\ \delta\end{matrix}\right)\left(\begin{matrix}\alpha&\beta\end{matrix}\right)=\begin{pmatrix}
1&-\varepsilon\beta\\
0&1
\end{pmatrix}\begin{pmatrix}
1&\varepsilon\beta\\
0&1
\end{pmatrix}=\begin{pmatrix}
1&0\\
0&1
\end{pmatrix}.
$$
Therefore, $\left(\begin{smallmatrix}\alpha&\beta\end{smallmatrix}\right)$ is a split monomorphism.

If $\gamma$ is a split epimorphism, then there is a homomorphism $\zeta\colon N\to L$ such that $\gamma\zeta=1$ and we have
$$
\left(\begin{matrix}\gamma\\ \delta\end{matrix}\right)\left(\begin{matrix}\varepsilon & \beta\end{matrix}\right)
\begin{pmatrix}
1&0\\
-\delta\varepsilon&1
\end{pmatrix}=\begin{pmatrix}
1&0\\
\delta\varepsilon&1
\end{pmatrix}\begin{pmatrix}
1&0\\
-\delta\varepsilon&1
\end{pmatrix}=\begin{pmatrix}
1&0\\
0&1
\end{pmatrix}.
$$
Therefore, $\left(\begin{smallmatrix}\gamma\\ \delta\end{smallmatrix}\right)$ is a split epimorphism.
Thus, $\left(\begin{smallmatrix}
f&0\\
0&1
\end{smallmatrix}\right)$ is irreducible. This proves (i)$\implies$(ii).

To show (iii)$\implies$(i), let $M\xra{a}L\xra{b}N$ be a factorization of $f$, and let $X$ be an $R$-module such that $\left(\begin{smallmatrix}
f&0\\
0&1
\end{smallmatrix}\right)\colon M\oplus X\to N\oplus X$ is irreducible.
We then have
$$
\begin{pmatrix}
f&0\\
0&1
\end{pmatrix}=\begin{pmatrix}
ba&0\\
0&1
\end{pmatrix}=\begin{pmatrix}
b&0\\
0&1
\end{pmatrix}\begin{pmatrix}
a&0\\
0&1
\end{pmatrix}.
$$
Then either $\left(\begin{smallmatrix}
b&0\\
0&1
\end{smallmatrix}\right)$ is a split epimorphism or $\left(\begin{smallmatrix}
a&0\\
0&1
\end{smallmatrix}\right)$ is a split monomorphism.
It is straightforward to see that either $b$ is a split epimorphism or $a$ is a split monomorphism.
Therefore, $f$ is irreducible.
\end{proof}

Let $n\geq 1$ be an integer. In the next lemma, $\diag(a_1,a_2,\dots,a_n)$ denotes the square matrix with $a_1, a_2,\dots,a_n$ on the main diagonal and zero everywhere else.

\begin{lem}\label{prop16052016a}
(a) For $1\le i\le n$ let $f_i\colon M_i\to N$ be a homomorphism of $R$-modules, and
assume that $\End_R(N)$ is a local ring.
If $$\left(\begin{matrix}f_1& f_2& \dots&f_n\end{matrix}\right)\colon M_1\oplus\cdots\oplus M_n\to N$$ is irreducible, then for all $1\le i\le n$ the $R$-homomorphism $f_i$ is irreducible.

(b) For $1\le i\le n$ let $h_i\colon M\to N_i$ be a homomorphism of $R$-modules, and
assume that $\End_R(M)$ is a local ring.
If $$\left(\begin{matrix}h_1&h_2&\dots&h_n\end{matrix}\right)^{tr}\colon M\to N_1\oplus\cdots\oplus N_n$$ is irreducible, then for all $1\le i\le n$ the $R$-homomorphism $h_i$ is irreducible.
\end{lem}

\begin{proof}
We only prove the first assertion; the second one is shown dually.
Also, we only prove that $f_1$ is irreducible, as the irreducibility of the other $f_i$ follow similarly.

First, suppose that $f_1$ is a split epimorphism. Then there is $g\colon N\to M_1$ such
that $f_1g=1$, and we have
$\left(\begin{matrix}f_1& f_2& \dots&f_n\end{matrix}\right)\left(\begin{matrix}g& 0& \dots&0\end{matrix}\right)^{tr}=1$.
This implies $\left(\begin{matrix}f_1& f_2& \dots&f_n\end{matrix}\right)$ is a split epimorphism, which contradicts
the assumption that it is irreducible. Hence, $f_1$ is not a split
epimorphism.

Next, suppose that $f_1$ is a split monomorphism. Since
$\End(N)$ is local, $N$ is indecomposable.\footnote{Note that we do not need Henselian property here.} Hence, $f_1\colon M_1\to N$ becomes an isomorphism,
and we have $\left(\begin{matrix}f_1& f_2& \dots&f_n\end{matrix}\right)\left(\begin{matrix}f_1^{-1}& 0& \dots&0\end{matrix}\right)^{tr}=1$.
This implies $\left(\begin{matrix}f_1& f_2& \dots&f_n\end{matrix}\right)$ is a split epimorphism, which contradicts
the assumption that it is irreducible. Hence, $f_1$ is not a split
monomorphism.

Now let $M_1\xrightarrow{\alpha}X\xrightarrow{\beta}N$ be a factorization of $f_1$.
We then have
$$
\left(\begin{matrix}f_1& f_2& \dots&f_n\end{matrix}\right)=\left(\begin{matrix}\beta\alpha& f_2& \dots&f_n\end{matrix}\right)=
\left(\begin{matrix}\beta& f_2& \dots&f_n\end{matrix}\right)\cdot\diag(\alpha,1,\dots,1).
$$
By assumption, either $\left(\begin{matrix}\beta& f_2& \dots&f_n\end{matrix}\right)$ is a split epimorphism or $\diag(\alpha,1,\dots,1)$ is a split monomorphism.
In the latter case, we easily see that $\alpha$ is a split monomorphism.
In the former case, we find homomorphisms $c\in\Hom_R(N,X)$ and $g_i\in\Hom_R(N,M_i)$ for $2\le i\le n$ such that the equality
\begin{equation}\label{eq16052016a}
\beta c+f_2g_2+\dots+f_ng_n=1
\end{equation}
holds in $\End_R(N)$.
Assume that $f_2g_2$ is a unit of $\End_R(N)$.
Then $f_2$ is a split epimorphism, so there is a map $d\colon N\to M_2$ such that $f_2d=1$.
Hence, we have $\left(\begin{matrix}f_1& f_2& \dots&f_n\end{matrix}\right)\left(\begin{matrix}0&d&0&\dots&0\end{matrix}\right)^{tr}=1$, which says that $\left(\begin{matrix}f_1& f_2& \dots&f_n\end{matrix}\right)$ is a split epimorphism, contrary to the assumption that it is irreducible.
Therefore, $f_2g_2$ is not a unit of $\End_R(N)$. Similarly, we can show that $f_3g_3,\dots,f_ng_n$ are not units of $\End_R(N)$.
Since $\End_R(N)$ is a local ring, from equation~\eqref{eq16052016a} we conclude that $\beta c$ is a unit. Hence, $\beta$ is a split epimorphism.
Thus $f_1$ is irreducible.
\end{proof}

For an $R$-module $M$ and for a positive integer $n$, by $M^{\oplus n}$ we denote the direct sum $\bigoplus_{i=1}^nM$.

\begin{lem}\label{prop16052016b}
Let $f\colon M\to N$ be a homomorphism of $R$-modules, and let $n$ be a positive integer.
If $f^{\oplus n}\colon M^{\oplus n}\to N^{\oplus n}$ is irreducible, then $n=1$.
\end{lem}

\begin{proof}
Suppose $n\ge2$.
Then $f^{\oplus n}\colon M^{\oplus n}\to N^{\oplus n}$ has a factorization
$$
M^{\oplus n}\xra{\diag(f,1,\dots,1)}N\oplus M^{\oplus n-1}\xra{\diag(1,f,\dots,f)}N^{\oplus n}.
$$
Since $f^{\oplus n}$ is irreducible, we conclude that either $\diag(f,1,\dots,1)$ is a split monomorphism or $\diag(1,f,\dots,f)$ is a split epimorphism.
If $\diag(f,1,\dots,1)$ is a split monomorphism (resp. $\diag(1,f,\dots,f)$ is a split epimorphism), then so is $f$, and so is $f^{\oplus n}$, contrary to its irreducibility.
Thus, we must have $n=1$.
\end{proof}

\section{Irreducible monomorphisms to free modules}\label{sec05062016a}

In this section we provide the proof of Theorem~\ref{thm16052016a}.

\begin{para}[Proof of Theorem~\ref{thm16052016a}, Part (a)]\label{para20170617a}
Let $\theta\colon I\to R$ be the inclusion map. Note that $\theta$ is neither a split monomorphism nor a split epimorphism.

Case 1: $I=\fm$. Suppose that there is a factorization $\m\xra{\alpha}M\xra{\beta}R$ of $\theta$.
Since $\theta$ is injective, so is $\alpha$ and
we have a commutative diagram
$$
\xymatrix{
0\ar[r]&\fm\ar[r]^{\alpha}\ar@{=}[d]&M\ar[r]\ar[d]^{\beta}&N\ar[r]\ar[d]^{\gamma}&0\\
0\ar[r]&\fm\ar[r]^{\theta}&R\ar[r]^{\pi}&k\ar[r]&0
}
$$
with exact rows.

If $\gamma\ne0$, then $\gamma$ is a surjection, so $\beta$ is also a surjection. Hence, $\beta$ is a split epimorphism.

If $\gamma=0$.
Then $\pi\beta=0$, and there exists a homomorphism $\delta:M\to\m$ such that $\beta=\theta\delta$.
Since $\theta$ is injective, we have $\delta\alpha=1$, whence $\alpha$ is a split monomorphism.
Consequently, $\theta$ is an irreducible homomorphism.

Case 2: General case. We can take a proper ideal $J$ of $R$ such that $\m\cong I\oplus J$.
Let $I\xrightarrow{\alpha}M\xrightarrow{\beta}R$ be a factorization of $\theta$, and
denote by $\theta'$ the inclusion map $J\to R$.
Then we have a factorization
$$
I\oplus J\xra{\left(\begin{matrix}
\alpha&0\\
0&1
\end{matrix}\right)}M\oplus J\xra{\left(\begin{matrix}\beta&\theta'\end{matrix}\right)}R
$$
of the map $\left(\begin{matrix}\theta&\theta'\end{matrix}\right)\colon I\oplus J\to R$, which is exactly the inclusion map $\m\to R$.
It follows from Case 1 that either
$\left(\begin{smallmatrix}
\alpha&0\\
0&1
\end{smallmatrix}\right)$ is a split monomorphism or
$\left(\begin{matrix}\beta&\theta'\end{matrix}\right)$ is a split epimorphism.

If $\left(\begin{smallmatrix}
\alpha&0\\
0&1
\end{smallmatrix}\right)$ is a split monomorphism, then we easily see that $\alpha$ is a split monomorphism.
If $\left(\begin{matrix}\beta&\theta'\end{matrix}\right)$ is a split epimorphism, then we can find elements $x\in M$ and $y\in J$ such that $\beta(x)+\theta'(y)=1$.
As $\theta'(y)=y$ is an element of the maximal ideal $\m$, the element $\beta(x)$ is a unit of $R$. Now define $\beta'\colon R\to M$ by $\beta'(r)=r(\beta(x))^{-1}x$ for every $r\in R$. It follows then that $\beta\beta'=1$.
Hence, $\beta$ is a split epimorphism and therefore, $\theta$ is irreducible, as desired. \qed
\end{para}

Recall that two homomorphisms $h\colon X\to Y$ and $h'\colon X'\to Y'$ of $R$-modules are
called {\em equivalent} if there exist isomorphisms $p\colon X\to X'$ and
$q\colon Y\to Y'$ such that $qh=h'p$. (It is easy to see that irreducibility
is preserved by equivalence.)

The next lemma enables us to replace an arbitrary monomorphism $M\to F$ of $R$-modules, where $F$ is free, with one whose image is contained in $\m F$.

\begin{lem}\label{n}
Let $\phi\colon M\to R^{\oplus m}$ be a monomorphism, where $M$ is an $R$-module and $m\geq 0$ is an integer.
Then there exist an integer $n$ with $0\leq n\leq m$, an $R$-module $N$, and a monomorphism $g\colon N\to R^{\oplus n}$ such that $\im(g)\subseteq \m R^{\oplus n}$ and such that $\phi$ is equivalent to the map
$$
\left(\begin{matrix}g&0\\0&1\end{matrix}\right)\colon N\oplus R^{\oplus m-n}\to R^{\oplus n}\oplus R^{\oplus m-n}
$$
where $1$ denotes the identity map of $R^{\oplus m-n}$. In this situation, $\phi$ is irreducible if and only if so is $\left(\begin{smallmatrix}
g&0\\
0&1
\end{smallmatrix}\right)$ if and only if so is $g$.
\end{lem}

\begin{proof}
There is a commutative diagram
$$\xymatrix{
&0\ar[d]&0\ar[d]&&\\
0\ar[r]&N\ar[r]^{g}\ar[d]^f&R^{\oplus n}\ar[r]^{\psi'}\ar[d]_s&C\ar[r]\ar@{=}[d]&0\\
0\ar[r]&M\ar[r]^{\phi}\ar[d]^{t'}&R^{\oplus m}\ar[r]^{\psi}\ar[d]_t&C\ar[r]&0\\
&R^{\oplus m-n}\ar@{=}[r]\ar[d]&R^{\oplus m-n}\ar[d]&&\\
&0&0&&
}
$$
with exact rows and exact columns such that $\psi'$ is minimal, that is, $\ker(\psi')\subseteq \frak mR^{\oplus n}$, for some integer $n$. (Note that here $M\cong N\oplus R^{\oplus m-n}$ such that $N$ does not have any free direct summand.) Therefore, $\im(g)\subseteq \frak mR^{\oplus n}$.
Note that the $R$-module homomorphism $\phi$ is equivalent to the map
$$
\left(\begin{matrix}g&0\\ 0&1\end{matrix}\right)\colon N\oplus R^{\oplus m-n}\to R^{\oplus n}\oplus R^{\oplus m-n}.
$$
Finally, the fact that $\phi$ is irreducible if and only if so is $\left(\begin{smallmatrix}
g&0\\
0&1
\end{smallmatrix}\right)$ if and only if so is $g$ follows from Lemma~\ref{p}.
\end{proof}

\begin{disc}\label{r}
By (the end of) Lemma~\ref{n}, the irreducibility of $\phi$ is
equivalent to the irreducibility of $g$. Thus, replacing $\phi$ with $g$, we will work with
irreducible maps $\phi\colon M\to R^{\oplus m}$ such that
$\im(\phi)\subseteq \m R^{\oplus m}$ in the rest of this section.
\end{disc}

\begin{prop}\label{pp}
Let $(R,xR)$ be a discrete valuation ring, and let $\theta\colon xR\to R$ be the inclusion map.
Let $\phi\colon M\to F$ be a monomorphism of $R$-modules with $F$ free such that $\im(\phi)\subseteq \m F$.
Then the following are equivalent:
\begin{enumerate}[\rm(i)]
\item
$\phi$ is irreducible.
\item
There is a commutative diagram
$$
\xymatrix{
M\ar[r]^{\phi}\ar[d]_{\eta}&F\ar[d]^{\pi}\\
xR\ar[r]^{\theta}&R
}
$$
such that $\eta$ and $\pi$ are isomorphisms.
\end{enumerate}
\end{prop}

\begin{proof}
(ii)$\implies$(i):
Assume that there is a commutative diagram
$$
\xymatrix{
M\ar[r]^{\phi}\ar[d]_{\eta}&F\ar[d]^{\pi}\\
xR\ar[r]^{\theta}&R
}
$$
such that $\eta$ and $\pi$ are isomorphisms.
It follows from Theorem~\ref{thm16052016a}(a) that $\theta$ is irreducible.
The above commutative diagram shows that $\phi$ is also irreducible.

(i)$\implies$(ii):
Since $R$ has global dimension one, $M$ is a free $R$-module.
Let $m:=\rank_RF$ and $n:=\rank_RM$.
As $\phi$ is injective, we have $m\ge n$.

Let $A$ be a representation matrix of $\phi$, which is an $m\times n$ matrix over $R$.
Since $\psi$ is minimal, each component of $A$ is an element of $xR$, whence there is another $m\times n$ matrix $B$ such that $A=xB$.
Hence, there is a factorization $R^{\oplus n}\xra{x}R^{\oplus n}\xra{B}R^{\oplus m}$ of $A$.
Irreducibility of $\phi$ implies that either $R^{\oplus n}\xra{x}R^{\oplus n}$ is a split monomorphism or $R^{\oplus n}\xra{B}R^{\oplus m}$ is a split epimorphism.
In the former case, we see that $R\xra{x}R$ is also a split monomorphism, which is impossible.
Hence, $R^{\oplus n}\xra{B}R^{\oplus m}$ is a split epimorphism. In particular, this says that $n\ge m$.
Therefore, we have $m=n$, and $R^{\oplus n}\xra{B}R^{\oplus m}$ is an isomorphism by~\cite[Theorem 2.4]{M}.
It follows that there is a commutative diagram
$$
\xymatrix{
M\ar[r]^{\phi}\ar[d]_{\lambda}&F\ar[d]^{\pi}\\
R^{\oplus n}\ar[r]^{x}&R^{\oplus n}
}
$$
such that $\lambda$ and $\pi$ are isomorphisms.
In view of Lemma~\ref{prop16052016b} we have $n=1$, and get a commutative diagram
$$
\xymatrix{
M\ar[r]^{\phi}\ar[d]_{\eta}&F\ar[d]^{\pi}\\
xR\ar[r]^{\theta}&R
}
$$
where $\eta$ is the composition $M\xra{\lambda}R\xra{x}R$. This completes the proof.
\end{proof}

\begin{disc}\label{disc20170618a}
(a) Let $0 \to L \xra{f} M \xra{g} N \to 0$
be a split exact sequence of $R$-modules. Then
there is an isomorphism $h\colon M\to L\oplus N$ such that $g=(\begin{matrix}0&1\end{matrix})h$.
Indeed, there is a homomorphism $\ell: M\to L$ such that $\ell f=1_L$. Set $h=(\begin{matrix}\ell&g\end{matrix})$.

(b) Let $F$ be a free $R$-module of rank $r$ and $a\colon F\to R^{\oplus r}$ be an isomorphism. Let $\pi\colon F\to R$ be an arbitrary surjection and set $b:=\pi a^{-1}$. Suppose that $p=(\begin{matrix}0&\ldots&0&1\end{matrix})\colon R^{\oplus r}\to R$ is the $r$-th projection. Since $b$ is a split epimorphism, applying part (a) to $M=R^{\oplus r}$ and $N=R$ with $g=b$ we obtain an automorphism $c\colon R^{\oplus r}\to R^{\oplus r}$ such that $pc=b$. Setting $q:=ca$ we have a commutative diagram
$$
\xymatrix{
F\ar[r]^{\pi}\ar[d]_{q}&R\ar@{=}[d]\\
R^{\oplus r}\ar[r]^{p}&R
}
$$
of $R$-modules in which $q$ is an isomorphism. This shows that $\pi$ and $p$ are equivalent.

\end{disc}

\begin{para}[Proof of Theorem~\ref{thm16052016a}, Part (b)]\label{para20170617b}
We prove the theorem step by step.

Step 1: Fix a proper submodule $D$ of $C:=\coker(\phi)$ and consider the pull-back diagram:
$$\xymatrix{
&&0\ar[d]&0\ar[d]&\\
0\ar[r]&M\ar[r]^{\alpha}\ar@{=}[d]&E\ar[r]\ar[d]_{\beta}&D\ar[r]\ar[d]&0\\
0\ar[r]&M\ar[r]^{\phi}&F\ar[r]^{\psi}\ar[d]&C\ar[r]\ar[d]&0\\
&&C/D\ar@{=}[r]\ar[d]&C/D\ar[d]&\\
&&0&0&
}
$$
Since $\phi$ is irreducible, either $\alpha$ is a split monomorphism or $\beta$ is a split epimorphism.
If $\beta$ is a split epimorphism, then $\beta$ must be an isomorphism, which implies $C=D$. This is a contradiction because we assumed that $D$ is a proper submodule of $C$.
Hence, $\alpha$ has to be a split monomorphism.

Step 2: Let $D$ be the submodule $\m C$ of $C$.
Since $\psi$ is minimal, the dimension of the $k$-vector space $C/\m C$ is equal to $r:=\rank_RF$.
From the middle column we observe that $E\cong \m^{\oplus r}$.
By Step 1, $M$ is isomorphic to a direct summand of $\m^{\oplus r}$.

Step 3: Let $D$ be a maximal submodule of $C$, that is, $C/D\cong k$.
Then it is observed that $E$ is isomorphic to $R^{\oplus r-1}\oplus\m$.
Hence, $M$ is isomorphic to a direct summand of $R^{\oplus r-1}\oplus\m$.

Step 4: Suppose that $R$ is isomorphic to a direct summand of $\m$.
Then it follows from~\cite[Corollary 1.3]{D} that $R$ is regular.
Hence, $R$ is a domain which forces $\m$ to be indecomposable.
Therefore, $\m$ is isomorphic to $R$, which means that $R$ is a discrete valuation ring.
In light of Proposition~\ref{pp}, we have now both of the conclusions (b1) and (b2) in case that $R$ is isomorphic to a direct summand of $\m$.

From now on, we assume that $R$ is not isomorphic to a direct summand of $\m$.

Step 5: Since $R$ is assumed to be Henselian, we can apply the Krull-Schmidt theorem. (See~\cite[Proposition 1.18]{yoshino}.)
It follows from Step 2 that $M$ does not contain a non-zero free summand; note here that $R$ is indecomposable as an $R$-module.
By Step 3 we see that $M$ is isomorphic to a direct summand of $\m$.
This shows Part (b1) of the theorem.

Step 6: To prove Part (b2), suppose that $M$ is indecomposable.
As $R$ is Henselian, $\End_R(M)$ is a local ring.
For each $1\leq i\leq r$ let $$p_i=(\begin{matrix}0&\ldots&0&1&0&\ldots&0\end{matrix})\colon R^{\oplus r}\to R$$ be the $i$-th projection (the $i$-th entry of $p_i$ is $1$). As we see in Remark~\ref{disc20170618a}(b), there is an isomorphism $q\colon F\to R^{\oplus r}$ such that the diagram
$$
\xymatrix{
F\ar[r]^{\pi}\ar[d]_{q}&R\ar@{=}[d]\\
R^{\oplus r}\ar[r]^{p_r}&R
}
$$
is commutative.

Now for each $1\leq i\leq r$ set $h_i:=p_i q \phi$ and note that $h_r=\pi\phi$. We have $q\phi=(\begin{matrix}h_1&h_2&\ldots&h_r\end{matrix})^{tr}$ and $q\phi$ is irreducible.
Lemma~\ref{prop16052016a} then implies that the composition $h_r=\pi\phi$ is irreducible.
In particular, it is not an epimorphism, so its image is contained in $\m$.
Thus, $\pi\phi$ has a factorization $M\xrightarrow{\eta}\m\xrightarrow{\theta}R$.
Irreducibility of $\pi\phi$ implies that $\eta$ is a split monomorphism.
We now have a commutative diagram
$$
\xymatrix{
M\ar[r]^{\phi}\ar[d]_{\eta}&F\ar[d]^{\pi}\\
\fm\ar[r]^{\theta}&R
}
$$
and the proof is completed. \qed
\end{para}

\begin{disc}
A natural question to ask is the following:
Let $I$ be an ideal of $R$, and for $1\le i\le n$ let $f_i\colon I\to R$ be an irreducible monomorphism.
Let $f=\left(\begin{matrix}f_1& f_2& \dots&f_n\end{matrix}\right)^{tr}\colon I\to R^{\oplus n}$.
When is $f$ irreducible?

Note that under the above assumptions, the map $f$ is not necessarily irreducible.
For example, if $n=2$ and $f_2=af_1$ for some unit element $a\in R$, then there is a factorization
$$
\xymatrix{
& R\ar[rd]^{\binom{1}{a}}\\
I\ar[rr]^{\left(\begin{smallmatrix}f_1\\ f_2\end{smallmatrix}\right)}\ar[ru]^{f_1} & & R^{\oplus2}
}
$$
which shows that $\left(\begin{smallmatrix}f_1\\f_2\end{smallmatrix}\right)\colon I\to R^{\oplus2}$ is not irreducible.
\end{disc}

\section{Irreducible epimorphisms from free modules}\label{sec05062016c}

This section is entirely devoted to the proof of Theorem~\ref{thm16052016b}.
The proof of Part (a) is also given in~\cite{LJ}. However, we include it for the convenience of the reader.

\begin{para}[Proof of Theorem~\ref{thm16052016b}]\label{para20170617c}
(a) If $M$ is free, then $\phi$ splits, which is a contradiction. Hence, $M$ is not free, and therefore $\Ext^1_R(M,k)\neq 0$. Let
$0\to k\to L\xra{h}M\to 0$ be a non-split short exact sequence of finitely generated $R$-modules in $\Ext^1_R(M,k)$.
Since $F$ is free there exists a homomorphism $g\colon F\to L$ such that $hg=\phi$. Since $\phi$ is irreducible and $h$ is not a split epimorphism, $g$ is a split monomorphism. Hence, there exist a finitely generated $R$-module $L'$ such that $L\cong F\oplus L'$ and a commutative diagram
$$\xymatrix{
&0\ar[d]&0\ar[d]&&\\
0\ar[r]&\ker(\phi)\ar[r]\ar[d]_t&F\ar[r]^\phi\ar[d]_g&M\ar[r]\ar@{=}[d]&0\\
0\ar[r]&k\ar[r]\ar[d]&L\ar[r]^h\ar[d]&M\ar[r]&0\\
&\coker(t)\ar[r]^{\cong}\ar[d]&L'\ar[d]&&\\
&0&0&&
}
$$
in which $\ker(\phi)\neq 0$. Since, $\dim_k\ker(\phi)+\dim_k\coker(\phi)=\dim_k k=1$, we conclude that $\dim_k\coker(t)=0$. Hence, $\coker(t)=0$ and $\ker(\phi)=k$, as desired.

(b1) The irreducibility of $\phi\iota$ follows from Lemma~\ref{prop16052016a}.
Hence, by Proposition~\ref{si} it is either surjective or injective.

(b2) By Theorem~\ref{thm16052016b}(a) there is an exact sequence
$0 \to k \xrightarrow{f} R \xrightarrow{\phi\iota} M \to 0$.
Let $x=f(\overline1)$.
Note that $x$ is a non-zero element of $\soc R$.
Let $y$ be any element of $\soc R$.
Then $g\colon k\to R$ given by $g(\overline1)=y$ is a monomorphism.
Considering the push-out diagram
$$
\xymatrix{
0\ar[r]&k\ar[r]^{f}\ar[d]_g&R\ar[r]^{\phi\iota}\ar[d]_{m}&M\ar[r]\ar@{=}[d]&0\\
0\ar[r]&R\ar[r]^{h}&E\ar[r]^{e}&M\ar[r]&0
}
$$
since $\phi\iota$ is irreducible, either $e$ is a split epimorphism or $m$ is a split monomorphism.

If $e$ is a split epimorphism, then
$h$ is a split monomorphism and there is a homomorphism $p\colon E\to R$ such that $ph=1$.
Hence, we have $g=phg=pmf$.
Note that $pm$ is an endomorphism of $R$.
Setting $a=pm(1)\in R$, we get
$$
y=g(\overline1)=pmf(\overline1)=pm(x)=ax.
$$
It follows then that $\soc R=(x)$.

If $m$ is a split monomorphism, then
there is a homomorphism $q\colon E\to R$ such that $qm=1$, and we have $f=qmf=qhg$.
Similarly as above, setting $b=qh(1)$, we get $x=by$.
Since $x$ is non-zero and $y$ is a socle element, $b$ must be a unit of $R$.
Hence, $y=b^{-1}x$ and $\soc R=(x)$.

So, if $e$ is a split epimorphism or $m$ is a split monomorphism, then we have $\soc R=(x)\cong k$, and $R$ has type one.
There is an isomorphism $\rho\colon R/\soc R\to M$ such that the diagram
$$
\xymatrix{
R\ar[rr]^{\phi\iota}\ar@{=}[d]&&M\\
R\ar[rr]^{\pi}&&R/\soc R\ar[u]^{\rho}
}
$$
commutes.
Using Theorem~\ref{thm16052016b}(a), we obtain a commutative diagram
$$
\xymatrix@C=3.6em@R=2.5em{
0\ar[r]&k\ar[r]^{f}\ar[d]_{\gamma}&R\ar[r]^{\pi}\ar[d]_{\iota}&R/\soc R\ar[r]\ar[d]^{\cong}_{\rho}&0\\
0\ar[r]&k\ar[r]&F\ar[r]^{\phi}&M\ar[r]&0
}
$$
with exact rows.
Since $\iota$ is injective, so is $\gamma$, and hence $\gamma$ is an isomorphism.
Five Lemma then shows that $\iota$ is also an isomorphism, whence $F$ has rank one. \qed
\end{para}

\begin{disc}
We work in the setting of Theorem~\ref{thm16052016b}. By Theorem~\ref{thm16052016b}(b1), the map $\phi\iota$ is either surjective or injective. As we see in Theorem~\ref{thm16052016b}(b2), in case that $\phi\iota$ is surjective, one can conclude that $R$ has type one, $F$ has rank one, and there is a commutative diagram
$$
\xymatrix{
F\ar[rr]^{\phi}&&M\\
R\ar[rr]^{\pi}\ar[u]^{\iota}&&R/\soc R\ar[u]_{\rho}
}
$$
such that $\iota$ and $\rho$ are isomorphisms and $\pi$ is the natural surjection.

Now, a natural question to ask is the following:

What can one conclude if in Theorem~\ref{thm16052016b}(b2) we replace the assumption ``$\phi\iota$ is surjective'' with ``$\phi\iota$ is injective for {\em all} split monomorphisms $\iota\colon R\to F$''?

Note that the assumption ``$\phi\iota$ is injective for {\em all} split monomorphisms $\iota\colon R\to F$'' is equivalent to saying that when we regard $\phi\colon F\to M$ as a surjection $$\left(\begin{matrix}\phi_1&\phi_2&\dots&\phi_m\end{matrix}\right)\colon R^{\oplus m}\to M,$$ all the components $\phi_i\colon R\to M$ are injective.
\end{disc}

\section{Irreducible homomorphisms and vanishing of (co)homology}\label{sec05062016d}

A commutative version of the Auslander-Reiten Conjecture~\cite{AR} states that if $M$ is an $R$-module with
$\Ext^{i}_R(M,M\oplus R)=0$ for all $i\geq 1$, then
$M$ is free. This conjecture has been proven affirmatively in some special cases; see for instance~\cite{AINSW, HL, HSV, JS, NSW, NY, S}.
The next theorem deals with this conjecture and also with a Tor version of it when the $R$-module $M$ admits irreducible homomorphisms described in Theorems~\ref{thm16052016a} and~\ref{thm16052016b}.

\begin{thm}\label{cor230915a}
Let $M$ be an $R$-module and consider the following conditions:
\begin{enumerate}[\rm(i)]
\item
$\Tor^R_i(M,M)=0$ for $i\gg 0$

\item
$\Ext_R^i(M,M)=0$ for $i\gg 0$

\item
$\pd_R M<\infty$

\item
$\id_R M<\infty$

\item
$R$ is a regular ring.
\end{enumerate}
Then (i)-(v) are equivalent under any of the following two conditions:
\begin{enumerate}[\rm(a)]
\item
$\frak m$ is indecomposable and there exists an irreducible monomorphism $M\to F$, where $F$ is free;

\item
there exists an irreducible epimorphism $F\to M$, where $F$ is free.
\end{enumerate}
Moreover, under the equivalent conditions (i)-(v) in these cases we have
\begin{equation}\label{eq270416a}\tag{$\ast$}
\pd_R M= \min\{n\in \mathbb{Z}\mid \Ext^n_R(M,M)\neq 0\ \text{and}\ \Ext^i_R(M,M)=0\ \text{for all $i>n$}\}.
\end{equation}
\end{thm}

Note that under the equivalent conditions in Theorem~\ref{cor230915a}, the equality~\eqref{eq270416a}
follows from~\cite[Proposition 2.4]{Diveris}. To prove Parts (a) and (b), we need the following proposition.


\begin{prop}\label{prop010516a}
Let $R^{\oplus m}\to M$ be an irreducible epimorphism of $R$-modules, where $m$ is a positive integer. If $\id_R M<\infty$,
then $R$ is a field.
\end{prop}

\begin{proof}
Since $M$ is non-zero, the ``Bass Conjecture''
Theorem (see for instance~\cite[9.6.2 and 9.6.4 (ii)]{bruns}) shows that $R$ is Cohen-Macaulay. By Theorem~\ref{thm16052016b}, Part (a) there
is an exact sequence
$0 \to k \to R^{\oplus m} \to M \to 0$.
In particular, this says that $\depth R=0$. Hence, $R$ is artinian and $M$ is
injective. Thus, $M$ is isomorphic to a direct sum of copies of the
injective hull $E$ of $k$, say $E^{\oplus n}$, where $n$ is a positive integer. The exact sequence
$0 \to k \to R^{\oplus m} \to E^{\oplus n} \to 0$ then
shows that
$m\ell=1+n\ell$,
where $\ell=\ell_R(R)=\ell_R(E)$. (Here $\ell_R$ denotes the length.) We now have $(m-n)\ell=1$, which
implies that $\ell=1$. This means that $R$ is a field.
\end{proof}

\begin{para}\label{para270416b}{(Proof of Theorem~\ref{cor230915a}).}
(a) Assume that $\frak m$ is indecomposable and $f\colon M\to F$ is an irreducible monomorphism of $R$-modules, where $F$ is free.
Note that conditions (i)-(v) and our assumptions in Part (a) of Theorem~\ref{cor230915a} are preserved under completion.
So, we replace $R$ by its completion in $\fm$-adic topology and assume that $R$ is complete (hence, $R$ is Henselian).

To show that conditions (i)-(v) are equivalent, it suffices to prove that each of (i) and (ii) implies (v). For this, let $g\colon N\to R^{\oplus n}$ be the map obtained by removing from $f$ the identity map of a maximal direct summand, as in Lemma~\ref{n}. Then the $R$-homomorphism
$g$ is irreducible and $N$ is isomorphic to a direct summand of $\m$ by Theorem~\ref{thm16052016a}(b1).
Since $N$ is non-zero and $\fm$ is indecomposable, $N\cong \fm$.

(i)$\implies$(v) From our Tor-vanishing assumption we have $\Tor_{i}^R(N,N)=0$ for all $i\gg 0$. Hence, $\Tor_{i}^R(\fm,\fm)=0$ for all $i\gg 0$.
This implies that $\Tor_{i}^R(k,k)=0$ for all $i\gg 0$, and therefore, $R$ is regular.

(ii)$\implies$(v) From our Ext-vanishing assumption we have $\Ext^{i}_R(N,N)=0$ for all $i\gg 0$. Hence, $\Ext^{i}_R(\fm,\fm)=0$ for all $i\gg 0$.
It follows then that $\Ext^{i}_R(k,\fm)=0$ for all $i\gg 0$, which implies that $\id_R \fm<\infty$. Therefore, $R$ is regular, as desired. (See Levin and Vasconcelos~\cite[Theorem 1.1]{LV}.)

(b) Assume that there exists an irreducible epimorphism $F\to M$, where $F$ is free. Using the short exact sequence $0\to k\to F\to M\to 0$ from Theorem~\ref{thm16052016b}(a), we see that (i) implies (v) and (ii) implies
(iv). Hence, it suffices to show that (iv) implies (v).
This follows from Proposition~\ref{prop010516a}. \qed
\end{para}

\begin{cor}\label{cor270416a}
Let $\depth R\geq 2$, and assume that there exists an irreducible monomorphism $M\to F$, where $F$ is free. Then all of the conditions (i)-(v) from Theorem~\ref{cor230915a} are equivalent.
\end{cor}

\begin{proof}
By~\cite[Corollary 3.3]{Ryo1}, we know that $\frak m$ is indecomposable. Now the assertion follows from Theorem~\ref{cor230915a}, Part (a).
\end{proof}

We conclude this section by proving the following result and one of its consequences. The proof is given in~\ref{para20171606a} below. In this theorem, $\gdim$ stands for the Gorenstein dimension of Auslander and Bridger~\cite{auslander}.

\begin{thm}\label{cor20170616a}
Let $M$ be an $R$-module. If $\fm$ is decomposable and $\gdim_R M<\infty$, then the following are equivalent.
\begin{enumerate}[\rm(i)]
\item
$\Tor^R_i(M,M)=0$ for $i\gg 0$

\item
$\Ext_R^i(M,M)=0$ for $i\gg 0$

\item
$\pd_R M<\infty$.
\end{enumerate}
\end{thm}

The following lemma is from~\cite[Theorem 6.3]{AINSW}. (See also~\cite[Corollary 4.4]{araya:rdf}.)

\begin{lem}\label{lem240416a}
Assume that $R\cong Q/(q)$ where $Q$ is a regular local ring with $q\in Q$. Then
$\Ext^i_R(M,M)=0$ for all $i\gg 0$ if and only if $\pd_R M<\infty$.
\end{lem}

\begin{para}[Proof of Theorem~\ref{cor20170616a}]\label{para20171606a}
We prove that if (i) or (ii) holds, then $\pd_R M<\infty$. Note that we can replace $R$ by its completion in the $\fm$-adic topology and assume that $R$ is complete.

Assume on the contrary that $\pd_R M=\infty$. Then by~\cite[Theorem A]{Ryo}, the ring $R$ is Gorenstein and is isomorphic to $Q/(q)$, where $Q$ is a regular local ring with $q\in Q$. Thus, if $\Tor_i^R(M,M)=0$ for all $i\gg 0$, it follows from~\cite[Theorem 1.9]{huneke} that $\pd_R M<\infty$, which is a contradiction.

Also, if $\Ext^i_R(M,M)=0$ for all $i\gg 0$, then by Lemma~\ref{lem240416a} we have $\pd_R M<\infty$, that is again a contradiction.

Hence, under the above assumptions we must have $\pd_R M<\infty$, as desired. \qed
\end{para}

\begin{disc}
After this paper was submitted, the authors were able to prove the equivalence of (i) and (iii) in Theorem~\ref{cor20170616a} without assuming that $M$ has finite G-dimension. The proof uses the notion of fiber products; see~\cite{taknass}.
\end{disc}

Following~\cite{Luo}, a finitely generated indecomposable $R$-module $M$ is called \emph{IG-projective} if $\gdim_R M=0$ and if $M$ admits either an irreducible epimorphism $F\to M$ or an irreducible monomorphism $M\to F$, in which $F$ is $R$-free. An example of IG-projective modules is the module $A/(X)$ over the local ring $A=\mathbb{K}[X]/(X^2)$ in which $\mathbb{K}$ is a field. Note that $A/(X)$ is not $A$-free.

As an immediate corollary of Theorems~\ref{cor230915a} and~\ref{cor20170616a}, we obtain the following result which is~\cite[Theorem 1.1]{LJ}.

\begin{thm}\label{thm240416a}
Let $M$ be an IG-projective $R$-module. Then
$\Ext_R^i(M,M)=0$ for all $i\geq 1$ if and only if $M$ is projective.
\end{thm}

\section*{Acknowledgments}
We are grateful to the referee for reading the paper very carefully and for giving
many valuable suggestions that improved the presentation of the paper significantly.

\end{document}